\documentclass[12pt,english]{article}
\usepackage[margin=1.5in]{geometry}

\usepackage{amssymb,amsbsy,latexsym}
\usepackage{amsmath}
\usepackage{url}

\usepackage[T1]{fontenc} 
\usepackage{fourier}
\usepackage{bm}

\usepackage{amscd,amsthm}

\usepackage{verbatim, comment}

\newtheoremstyle{theorem}{1em}{1em}{\slshape}{0pt}{\bfseries}{.}{ }{}
\theoremstyle{theorem}
\newtheorem{theorem}{Theorem}
\newtheorem{conjecture}{Conjecture}
\newtheorem{corollary}[theorem]{Corollary}
\newtheorem{proposition}[theorem]{Proposition}
\newtheorem{lemma}[theorem]{Lemma}
\newtheorem{definition}[theorem]{Definition}
\newtheorem*{lemma*}{Lemma}
\theoremstyle{remark}

\newcommand{\E}{\mathop{\mathbb{E}}}

\newcommand{\R}{\mathbb R}
\newcommand{\eps}{\varepsilon}
\newcommand{\ip}[2]{\langle #1,#2\rangle}
\newcommand{\abs}[1]{\left|#1\right|}
\newcommand{\norm}[1]{\left\|#1\right\|}

\newcommand{\wg}{w}
\DeclareMathOperator{\codim}{codim}
\DeclareMathOperator{\spanop}{span}
\DeclareMathOperator{\vol}{vol}
\DeclareMathOperator{\vr}{vr}
\DeclareMathOperator{\conv}{conv}

\usepackage{calrsfs} 
\DeclareMathAlphabet{\pazocal}{OMS}{zplm}{m}{n}

\usepackage[hidelinks]{hyperref}
\makeatother

\title{Optimal Vector Balancing for Zonotopes}

\author{Victor Reis\thanks{Microsoft Research, Redmond. Email:
\href{mailto:victorol@microsoft.com}{\texttt{victorol@microsoft.com}}.}}
\date{}
\begin{document}

\maketitle

\begin{abstract}
A \textit{zonotope} is a linear image of the cube $[-1,1]^m$ for some $m \in \mathbb{N}$. We show that there is a universal constant $C>0$ such that, for every
zonotope $Z\subset \mathbb{R}^d$ and vectors
$v_1,\dots,v_n\in Z$, there are signs $x_1,\dots,x_n\in\{-1,1\}$ with
\[
  \sum_{i=1}^n x_i v_i \in C\sqrt d\, Z.
\] This resolves a 2002 question of Schechtman
and generalizes Spencer's six standard deviations theorem,
which corresponds to the case $Z=[-1,1]^d$.
\end{abstract}

\section{Introduction}

Spencer's \textit{six standard deviations theorem}~\cite{Spencer1985} states that for every choice of vectors $v_1, \dots, v_d \in [-1,1]^d$, there are signs $x_1, \dots, x_d \in \{-1,1\}$ so that 
\[ \Big\|\sum_{i=1}^d x_i v_i\Big\|_\infty  \le 6\sqrt d.
\]
In 2002, Schechtman asked whether a $C \sqrt{d}$ bound holds for vectors in an
arbitrary \textit{zonotope} $Z = A^\top [-1,1]^m \subset \mathbb{R}^d$. This question appears as the first problem in \emph{Open problems on embeddings
of finite metric spaces}~\cite[Problem~2.1]{MatousekNaor2003}; see
also~\cite[Problem~1.1]{AIMFourierConvex2007}.  Our main theorem answers it affirmatively.
\begin{theorem}\label{thm:main-vb}
There is a universal constant $C>0$ such that the following holds. Let
$Z\subset \R^d$ be a zonotope. Then, for every integer $n \in [1,d]$ and all
$v_1,\ldots,v_n\in Z$, there exist polynomial-time computable signs
$x_i\in\{-1,1\}$ such that
\[
  \sum_{i=1}^n x_i v_i
  \in C\sqrt{n\log\Big(\frac{2d}{n}\Big)}\,Z.
\]
\end{theorem}

Previously, the best known bound for $n = d$ was $O(\sqrt d\,\log\log\log d)$~\cite{HeckReisRothvoss2023}. For the cube $Z=[-1,1]^d$, this recovers Spencer's bound and is optimal up to universal constants for every $n$ and $d$~\cite[Theorem~71]{Reis2023Thesis}. A standard linear-algebraic argument shows that for $n>d$, the bound is the same as the $n=d$ bound up to a factor of two~\cite[Theorem~3]{Reis2023Thesis}.

There is a natural route to this statement through zonotope sparsification.
Talagrand proved that there exists a zonotope $\widetilde{Z} = \tilde{A}^\top [-1,1]^{\widetilde{m}}$ with $\widetilde{m} = O(\frac{d}{\eps^2}\log d)$ and $\widetilde{Z} \subseteq Z \subseteq (1+\eps) \widetilde{Z}$~\cite{Talagrand1990}.  A 1987 problem of Schechtman~\cite[Problem~7]{Schechtman1987} (see also~\cite[Problem~1.12]{AIMFourierConvex2007}) asks whether the logarithmic factor can be removed: does every
$d$-dimensional zonotope admit such an approximation with $\widetilde{m} = d \cdot C(\eps)$?  

Spencer~\cite{Spencer1985} proved more generally that for vectors $u_1, \dots, u_n \in [-1,1]^m$, there are signs $x_1, \dots, x_n \in \{-1,1\}$ so that 
\[ \Big\|\sum_{i=1}^n x_i u_i\Big\|_\infty  \le 11\sqrt {n \log\Big(\frac{2m}{n}\Big)}.
\]

Thus assuming an approximation $\widetilde{Z}$ with $\widetilde{m} = d \cdot C(\eps)$ exists for constant $\eps$, given $v_1, \dots, v_d \in Z$, we may write $v_i = (1+\eps) \tilde{A}^\top u_i$  and apply Spencer's theorem to $u_i \in [-1,1]^{\widetilde{m}}$, yielding $\|\sum_{i=1}^{d} x_i u_i\|_\infty \le  11\sqrt {d \log\Big(\frac{2\widetilde{m}}{d}\Big)} = 11\sqrt{1 + \log C(\eps)} \sqrt{d}$ and consequently $\sum_{i=1}^d x_i v_i  \in 11(1+\eps)\sqrt{1 + \log C(\eps)} \sqrt{d} Z$. Therefore, this would yield an affirmative answer to Schechtman's zonotope vector balancing question. Alas, the bound  $\widetilde{m} = O(\frac{d}{\eps^2}\log d)$ due to Talagrand remains the state of the art for zonotope sparsification, so this approach currently only leads to an $O(\sqrt{d \log \log d})$ vector balancing bound.

\subsection{Proof overview}

We reduce the problem to a volumetric bound on quotients of sections of the \(\ell_1\) ball. In Section 2, we recall the \textit{Lewis position} of the polar of a zonotope, which we denote by $K_1$, along with other concepts that will be needed later in the proof. In Section 3, we apply \textit{Gordon's escape theorem} to find sections of $K_1$ with small radius. In Section 4, we generalize this to sections of linear images $MK_1$. For this, we will need to interpolate between $K_1$ and the Euclidean ball. To deal with the linear maps, we use the fact that linear maps preserve coverings, and bound entropy numbers via \textit{Carl's inequality}. In Section 5, we pass from the Lewis position to arbitrary sections of the \(\ell_1\) ball, apply Carl's inequality again to bound entropy numbers of quotients of sections, and derive the volumetric estimates.
Finally, in Section 6 we identify the convex body associated with balancing vectors on a zonotope with the quotient of a section of the \(\ell_1\) ball and construct a partial coloring, which can be iterated to achieve the optimal vector balancing bound.

We use $A \lesssim B$ to denote that there exists a universal constant $C$ with ${A \le C \cdot B}$. For $p\in[1,\infty]$, $B_p^d$
is the unit ball of $\ell_p^d$, and $\log$ denotes the binary logarithm.

\subsection{Other related work}\label{sec:related-work}

Theorem~\ref{thm:main-vb} was previously known in the special case where $Z = B^d_p$, $p \in [2, \infty]$, for which the tight bound is $\Theta(\sqrt{n \min\{p, \log(2d/n)\}})$~\cite{ReisRothvoss2022}. While these are not strictly zonotopes for $p < \infty$, they are \textit{zonoids}, which can be approximated arbitrarily well by a zonotope, so Theorem~\ref{thm:main-vb} still applies.

Suppose the vectors satisfy $v_1, \dots, v_n \in \{0,1\}^d$ with at most $t$ ones each.  Beck and Fiala~\cite{IntegerMakingTheorems-BeckFiala81}
proved, using a linear-algebraic argument, that there are signs with $\|\sum_{i=1}^n x_i v_i\|_\infty \le 2t$ and conjectured an $O(\sqrt{t})$ bound. A recent work of Bansal and Jiang~\cite{BansalJiang2025} resolved the Beck--Fiala conjecture for $t \ge \log^2 n$ and obtained a bound of $(\sqrt{t} + \sqrt{\log n}) \cdot (\log \log n)^{O(1)}$ for smaller values of $t$; when $t \le \sqrt{\log n}$, the best bound is still $2t-\log^*(t)$~\cite{bukh_2016}. 

A stronger conjecture is due to Koml\'os, who asked whether for $v_1, \dots, v_n \in B_2^d$ there exist signs with $\|\sum_{i=1}^n x_i v_i\|_\infty \le O(1)$. This would also generalize Spencer's theorem, and the best known bound of $(\log^{1/4} n) \cdot (\log \log n)^{O(1)}$ is also due to Bansal and Jiang~\cite{BansalJiang2025}.

Another open generalization of Spencer's theorem is the \textit{Matrix Spencer conjecture}~\cite{Zouzias2012, MekaBlog2014} which asks whether for symmetric matrices $A_1,\ldots,A_n \in \R^{n \times n}$ with eigenvalues in $[-1,1]$, there are
signs $x \in \{-1,1\}^n$ so that the maximum singular value of $\sum_{i=1}^n x_iA_i$ is at most $O(\sqrt{n})$. This conjecture has only been proved under the additional assumption that the matrices are
block-diagonal with constant-size blocks~\cite{DBLP:conf/stoc/DadushJR22}, or have rank at most $\frac{n}{\log^3 n}$~\cite{BansalJiangMeka-MatrixSpencerUpToPolylog-2022-Arxiv}.

Regarding the zonotope sparsification problem, Cohen and Peng~\cite{CohenPeng2015} made Talagrand's $O(\frac{d}{\eps^2}\log d)$ sparsity bound algorithmic. Linear-size sparsification has only recently been obtained for $O(1)$-modular zonotopes~\cite{EisenbrandRothvossRussoSkorupinski2026}, and Talagrand's bound remains the state of the art in general.

\section{Preliminaries}\label{sec:prelim}

A symmetric convex body $K \subset \R^d$ is a compact convex set with nonempty interior so that $K=-K$.  Its associated norm is $\|x\|_K:=\inf\{t>0:x\in tK\}$, its polar is $K^\circ:=\{y\in\R^d: |\ip{x}{y}|\le 1 \forall x\in K\}$, and $\|y\|_{K^\circ}:=\sup_{x\in K}|\ip{x}{y}|$ is the polar norm. When a symmetric convex set \(K\) is contained in a proper subspace, its norm and its polar are also considered to be taken inside that subspace.

\subsection{Lewis position, John decomposition, and volume ratio}\label{subsec:lewis-position}

Throughout, fix a full-dimensional zonotope $Z = A^\top B^m_\infty \subset \mathbb{R}^d$, where $A\in\R^{m\times d}$ has rank $d$, and all rows $A_1,\ldots,A_m$ are nonzero. We normalize the polar of $Z$ so that the zonotope is in \textit{Lewis position}~\cite{Lewis78}. Observe that
\[
  Z^\circ =\Big\{x\in\R^d:\sum_{i=1}^m |\ip{x}{A_i}|\le1\Big\}.
\]

\begin{theorem}[{\cite[Lemma~9]{Ball1991}}]\label{thm:lewis-rep}
There are unit vectors
$u_1,\ldots,u_m\in \partial B_2^d$, weights $c_1,\ldots,c_m>0$, and an invertible linear map $T:\R^d\to\R^d$ such that
\begin{equation}\tag{$\dagger$}\label{eq:isotropic}
  K_1 := TZ^\circ
  =\Big\{x\in\R^d:\sum_{i=1}^m c_i|\ip{x}{u_i}|\le1\Big\}, \qquad  \sum_{i=1}^m c_i\,u_i u_i^\top=I_d.
  \end{equation}
\end{theorem}
In Sections 3 and 4, we put $Z$ in Lewis position and denote $K_1 := Z^\circ$.

This normalization is slightly different from the one in~\cite{HeckReisRothvoss2023}, where \textit{normalized zonotopes} are defined using Lewis weights.  One can check that if a zonotope $Z$ is in the above Lewis position, then for every $\eps>0$ there is a normalized zonotope $\widetilde Z$ with
$(1-\eps)\widetilde Z\subset Z\subset (1+\eps)\widetilde Z$.  We will not use this, as the Lewis position is more convenient. 

For $p \in [1,2]$, define the interpolation bodies
\[
K_p:=\{x\in\R^d:\sum_{i=1}^m c_i |\ip{x}{u_i}|^p\le1\}.
\]
Thus $K_2=B_2^d$. The next two lemmas follow directly from H\"older's inequality.

\begin{lemma}\label{lem:weighted-lewis-l1-lp-comparison}
Let $p \in (1,2]$, and put $q=p/(p-1)$.  Then, for every $x\in\R^d$,
\[
  \|x\|_{K_1}\le d^{1/q}\|x\|_{K_p}.
\]

\end{lemma}

\begin{lemma}\label{lem:weighted-lewis-local-interpolation}
Let $p \in (1,2]$, and put $q=p/(p-1)$.  Then, for every $x\in\R^d$,
\[
  \|x\|_{K_p}
  \le
  \|x\|_{K_1}^{1-2/q}\|x\|_2^{2/q}.
\]
\end{lemma}

We also use the following inclusions:

\begin{lemma}\label{lem:K1-subset-B2} One has $\frac1{\sqrt d}B_2^d\subseteq K_1\subseteq B_2^d$.
\end{lemma}

\begin{proof} The first inclusion follows from Lemma~\ref{lem:weighted-lewis-l1-lp-comparison} with $p=2$. Take $x \in K_1$. By \eqref{eq:isotropic},
\[
  \|x\|_2^2=\sum_{i=1}^m c_i |\ip{x}{u_i}|^2.
\]
Since $u_i\in \partial B_2^d$, we have $|\ip{x}{u_i}|\le\|x\|_2$, so
\[
  \|x\|_2^2
  \le \|x\|_2\sum_{i=1}^m c_i |\ip{x}{u_i}| \le \|x\|_2.
\]
Then $\|x\|_2\le 1$, so $K_1\subseteq B_2^d$.
\end{proof}

For a different convex body, we also use the symmetric John position~\cite{John1948}.

\begin{theorem}
\label{thm:john-loewner}
Let $K\subset \R^n$ be a symmetric convex body whose
maximal-volume inscribed ellipsoid is $B_2^n$.  There are unit vectors
$u_1,\ldots,u_N\in K \cap K^\circ\cap \partial B_2^n$ and weights
$c_1,\ldots,c_N>0$ such that
\[
  \sum_{i=1}^N c_i u_i u_i^\top=I_n.
\]
\end{theorem}

For a full-dimensional convex body $K\subset \R^n$, let
$\pazocal E_K$ denote the maximal-volume ellipsoid contained in $K$, and define
\(\vr(K):=\left(\frac{\vol_n(K)}{\vol_n(\pazocal E_K)}\right)^{1/n}\).
This quantity is invariant under invertible linear transformations, and so it will later be useful to assume John position when bounding the volume ratio of a convex body.

\subsection{Entropy and Gelfand numbers}

For convex bodies $K,L\subset \R^d$, let $N(K,L)$ be the least number of translates of $L$ needed to cover $K$, and define the $k$th dyadic \textit{entropy number} by \[ e_k(K,L):=\inf\{t>0:N(K,tL)\le 2^{k-1}\}.\]  We abbreviate $e_k(K):=e_k(K,B_2^d)$ and use this quantity to upper bound volume. We in turn upper bound entropy numbers in terms of \textit{section radii} known in the operator theory literature as Gelfand numbers:

\begin{definition}\label{def:section-radius}
Let $K,L\subseteq\R^d$ be symmetric convex bodies and denote
$H:=\spanop K$.  For $k \in \mathbb{N}$, define
\[
  r_k(K,L):=\inf\left\{
  \sup_{x\in K\cap F}\|x\|_L:
  F\subseteq H\text{ a subspace},\ \codim_H F<k
  \right\}.
\]
\end{definition}
We abbreviate $r_k(K):=r_k(K,B_2^d)$. Clearly, $r_k (K,L)$ is decreasing in $k$, and if $k>\dim H$, then $r_k(K,L)=0$, since one may take $F=\{\mathbf{0}\}$. The infimum is attained: for each $k$ the corresponding
Grassmannian in $H$ is compact, and the objective is continuous on
that Grassmannian.

\textit{Carl's inequality} upper bounds entropy numbers in terms of section radii; see~\cite{Carl1981} and Pisier's textbook~\cite[Chapter~5]{Pisier1999}. We state the general version, although we only need it for $\alpha = 1$ and $L = B^d_2$.

\begin{theorem}\label{thm:carl-covering}
For every $\alpha>0$ there is a constant $C_\alpha$ such that the following
holds.  Let $K,L$ be symmetric convex bodies in $\R^d$, and let $k\ge1$ be an
integer.  Then
\[
  \max_{j \in [1,k]} j^\alpha e_j(K,L)
  \le C_\alpha\max_{j \in [1,k]} j^\alpha r_j(K,L).
\]
\end{theorem}

We use \textit{Gordon's escape theorem}~\cite{Gordon1988} to construct sections of small radii. For a bounded set $S\subseteq\R^n$, its \textit{(Gaussian) width} is $\wg(S):=\E_{g\sim N(0,I_n)}\sup_{x\in S}\ip{g}{x}$.

\begin{theorem}[{\cite[Theorem~9.3.4]{Vershynin2018}}]\label{thm:escape}
Let $S\subseteq \partial B_2^n$ be any set, and let $G$ be an $m\times n$ random matrix with independent standard Gaussian entries.

If $m\gtrsim \wg(S)^2$, then the random subspace $F:=\ker G$ satisfies $S\cap F=\emptyset$ with probability at least $1-2\exp(-\Omega(m))$.
In particular, there exists a universal constant $m_{\mathrm E} \in \mathbb{N}$ so that for
$m\ge m_{\mathrm E}\max\{1,\wg(S)^2\}$ one can find a subspace
$F \subseteq \mathbb{R}^n$ with $\codim F\le m$ and $S\cap F=\emptyset$.
\end{theorem}

We now turn Gaussian width into a section radius bound. This is the
geometric form of Carl--Pajor's Gelfand-number estimate~\cite[Lemma~2.1]{CarlPajor1988},
also recorded in Pisier's textbook~\cite[Chapter~5, Theorem~5.8]{Pisier1999}. A self-contained proof using Theorem~\ref{thm:escape} is included in Appendix~\ref{app:covariance-width-section-proofs}.

\begin{proposition}\label{prop:covariance-width-section}
Let $K\subseteq\R^n$ be a symmetric convex body and $k \in \mathbb{N}$. Then
\[
  r_k(K)
  \lesssim \frac{\wg(K)}{\sqrt{k}}.
\]
\end{proposition}

\subsection{Covariance domination and Gaussian width}

In this subsection, we translate tools from operator theory to convex geometric language. For completeness, we include the short proofs in Appendix~\ref{app:covariance-width-section-proofs}.

For a convex symmetric body $K\subseteq\R^d$, define $\pazocal M(K):=\conv\{yy^\top:y\in K\}$ and
\[
  \pazocal{B}_n(K):=
  \left\{M \in \mathbb{R}^{n \times d}:
  M^\top M\preceq A
  \text{ for some }A\in\pazocal M(K^\circ)
  \right\}.
\]
Here $B\preceq A$ means that $A-B$ is positive semidefinite. $\pazocal{B}_n(K)$ is the unit ball of the
absolutely $2$-summing norm introduced by
Pietsch~\cite{Pietsch1979}; see~\cite[Chapter~1]{Pisier1999}. 

\begin{lemma}\label{lem:Bnorm-convex}
$\pazocal{B}_n(K)$ is convex for any symmetric convex body $K\subseteq\R^d$ and $n \in \mathbb{N}$.
\end{lemma}

\begin{definition}\label{def:gamma-covariance}
Let $K\subseteq\R^d$ be a symmetric convex body.  Define
\[
  \Gamma(K):=
  \sup_{A\in\pazocal M(K^\circ)}
  \E_{g\sim N(0,A)}\|g\|_{K^\circ}.
\]
\end{definition}
The supremum is attained since $\pazocal M(K^\circ)$ is compact. In the language of operator theory, this is referred to as the Gaussian type-$2$ constant~\cite[Chapter~9]{LedouxTalagrand1991}.

\begin{lemma}\label{lem:Bnorm-dominates-radius-one}
Let $K\subseteq\R^d$ be a symmetric convex body, and let
$M:\R^d\to\R^n$ be linear. Then
\[
  r_1(MK)\le \norm{M}_{\pazocal{B}_n(K)}.
\]
\end{lemma}

The next lemma bounds the width of a linear image~\cite[Lemma~1]{DavisMilmanTomczak1981}.

\begin{lemma}\label{lem:width-covariance-domination}
Let $K\subseteq\R^d$ be a symmetric convex body, and let $M:\R^d\to\R^n$ be
linear.  Then
\[
  \wg(MK)\le \norm{M}_{\pazocal{B}_n(K)}\Gamma(K).
\]
\end{lemma}

Finally, we use a bound on the Gaussian width from a cover. 

\begin{lemma}\label{lem:cover-to-width}
Let $K\subseteq r B_2^n$ be bounded, and let $L\subseteq\R^n$ be symmetric and convex.  If
\[
  N(K,tL) \le 2^{k-1},
\]
then
\[
  \wg(K)\lesssim r\sqrt k +t\,\wg(L).
\]
\end{lemma}

\subsection{Partial coloring from hereditary volume bounds}

For $S\subseteq[n]$, write $\R^S:=\{x\in\R^n:x_i=0 \ \forall i\notin S\}$.  For a convex body $K\subseteq\R^n$, define $K_S:=K\cap\R^S$.
We regard $K_S$ as a convex body in the coordinate space $\R^S$ and use
$\vol_{|S|}$ for volume in that space.  We also use the convention
$\vol_0(K_\emptyset)=1$.

Let $\gamma_n$ denote the standard Gaussian measure on $\R^n$.  We use two results of Reis and Rothvoss.  The first, \cite[Theorem~7]{ReisRothvoss2022}, says that hereditary lower bounds on coordinate-section volumes imply a Gaussian measure lower bound.

\begin{theorem}\label{thm:rr-coordinate-volume}
For every symmetric convex body $K \subset \mathbb{R}^n$,
\[
  \gamma_n(K)^{1/n}
  \gtrsim
  \left(\min_{S\subseteq[n]}\vol_{|S|}(K_S)\right)^{1/n}.
\]
\end{theorem}

The second produces a shifted partial coloring~\cite[Theorem~6]{ReisRothvoss2022}.

\begin{theorem}\label{thm:rr-shifted-pc}
For every $\alpha,\beta,\eta>0$ there is a constant
$C=C(\alpha,\beta,\eta)>0$ such that the following holds.  Let
$K\subseteq\R^n$ be a symmetric convex body with $\gamma_n(K)\ge e^{-\alpha n}$. Let $y\in[-1,1]^n$, and let $H\subseteq\R^n$ be a subspace with
$\dim(H)\ge\beta n$.  Then there is $x\in CK\cap H$ such that
$x+y\in[-1,1]^n$ and $\abs{\{i\in[n]:(x+y)_i\in\{\pm1\}\}}
  \ge (\beta-\eta)n$. Moreover, it may be constructed in polynomial time.
\end{theorem}

We use the following consequence. See also~\cite[Lemma~8]{DadushNikolovTalwarTomczak2018}.

\begin{corollary}\label{thm:volume-pc}
Let $n\ge1$, and let $K\subseteq\R^n$ be a symmetric convex body such that
\[
  \vol_{|S|}(K_S) \ge 1
  \qquad\text{for every }S\subseteq[n].
\]
Then for every $y\in[-1,1]^n$ there is a polynomial-time computable $y'\in[-1,1]^n$ such that $\abs{\{i\in[n]:y'_i\in\{\pm1\}\}}\ge \frac n2$
and $\|y-y'\|_K\lesssim 1$. 
\end{corollary}

\begin{proof}
The volume assumption and Theorem~\ref{thm:rr-coordinate-volume} imply $\gamma_n(K)^{1/n}\gtrsim1$. Apply Theorem~\ref{thm:rr-shifted-pc} to obtain $x\in CK$ such that $y':=x+y$ belongs to $[-1,1]^n$ and has at least $n/2$ coordinates
in $\{\pm1\}$. Then $\|y-y'\|_K=\|x\|_K\lesssim1$.
\end{proof}

Finally, we use the inverse Santal\'o inequality of Bourgain and Milman~\cite{BourgainMilman1987}.

\begin{theorem}\label{thm:inverse-santalo}
For every symmetric convex body $K \subset \mathbb{R}^n$,
\[
  \vol_n(K)^{1/n}\vol_n(K^\circ)^{1/n}
  \gtrsim
  \frac1n.
\]
\end{theorem}

\section{A section radius bound for $K_1$}

We derive an upper bound on the section radii $r_k (K_1)$ of the body defined in \eqref{eq:isotropic} via Gordon's escape theorem (Theorem~\ref{thm:escape}). First, we need to upper bound its width.

\begin{proposition}\label{prop:localized-width}
For every $r \ge 0$,
\[
  \wg\bigl(K_1\cap rB_2^d\bigr)
  \lesssim \sqrt{\log(1+dr^2)}.
\]
\end{proposition}

\begin{proof}
Let $x\in K_1\cap rB_2^d$. For $g\sim N(0,I_d)$, \eqref{eq:isotropic} gives
\[
  \ip{g}{x}=\sum_{i=1}^m c_i\ip{g}{u_i}\ip{x}{u_i}.
\]
Let $q\ge2$ be chosen later and put $p=q/(q-1)$.  Then
Lemma~\ref{lem:weighted-lewis-local-interpolation} gives
\[
  \|x\|_{K_p}\le r^{2/q}.
\]
Thus H\"older's inequality yields
\[
  |\ip{g}{x}|
  \le
  \left(\sum_{i=1}^m c_i |\ip{x}{u_i}|^p\right)^{1/p}\left(\sum_{i=1}^m c_i |\ip{g}{u_i}|^q\right)^{1/q}
  \le
  r^{2/q}\left(\sum_{i=1}^m c_i |\ip{g}{u_i}|^q\right)^{1/q}.
\]
By concavity of $t\mapsto t^{1/q}$, this gives
\[
  \wg(K_1\cap rB_2^d) = \E \sup_{x \in K_1\cap rB_2^d} |\ip{g}{x}|
  \le
  r^{2/q}\left(\sum_{i=1}^m c_i\,\E|\ip{g}{u_i}|^q\right)^{1/q}.
\]
For $u_i\in \partial B_2^d$, $\ip{g}{u_i} \sim N(0,1)$, so
$(\E|\ip{g}{u_i}|^q)^{1/q}\lesssim \sqrt q$.  Since $\sum_{i=1}^m c_i = d$ from~\eqref{eq:isotropic},
\[
  \wg(K_1\cap rB_2^d)
  \lesssim \sqrt q\,(dr^2)^{1/q}.
\]
If $t:=dr^2 < 4$, take $q=2$ and use $t \lesssim \log(1+t)$ on this range. If
$t \ge 4$, take $q=\log t$, so $t^{1/q}=2$.  In either case, the above inequality yields
\[
  \wg(K_1\cap rB_2^d)
  \lesssim \sqrt{\log(1+dr^2)}. \qedhere
\]
\end{proof}

\begin{corollary}
\label{cor:localized-width-escape-scale}
For $k \in [1,d]$ and $\lambda\ge1$, take $r^2 = \lambda\frac{\log(2d/k)}{k}$. Then 
\[
  \wg\bigl((\tfrac{1}{r} K_1) \cap \partial B_2^d\bigr)^2
  \lesssim \frac{k\log(2\lambda)}{\lambda}.
\]
\end{corollary}

\begin{proof}
By Proposition~\ref{prop:localized-width} and homogeneity,
\[
  \wg\bigl((\tfrac{1}{r} K_1) \cap \partial B_2^d\bigr)^2 = \wg\bigl(r^{-1} (K_1 \cap r\partial B_2^d)\bigr)^2 \le \wg\bigl(r^{-1}(K_1\cap rB_2^d)\bigr)^2
  \lesssim \frac{\log(1+dr^2)}{r^2}.
\]
Put $\alpha:=\log(\tfrac{2d}{k}) \ge1$. By assumption, $r^2 = \tfrac{\lambda\alpha}{k}$, so
\[
   \wg\bigl((\tfrac{1}{r} K_1) \cap \partial B_2^d\bigr)^2 \lesssim \frac{\log(1+dr^2)}{r^2}
  =
  \frac{k}{\lambda} \cdot \alpha^{-1}
  \log\Big(1+\tfrac{\lambda d  \alpha}{k}\Big).
\]
Moreover, since $\lambda\ge1$ and $\alpha\ge1$,
\[
  \alpha^{-1}
  \log\Big(1+\tfrac{\lambda d  \alpha}{k}\Big)
  =\alpha^{-1}\log(1+\lambda 2^{\alpha-1}\alpha)
  \le \alpha^{-1}\log(2\lambda \cdot 2^{\alpha-1}\alpha)
  \lesssim \log(2\lambda). \qedhere
\]
\end{proof}

\begin{theorem}\label{thm:special-lewis}
For every integer $k \in [1, d]$,
\[
  r_k\bigl(K_1\bigr)
  \lesssim \sqrt{\frac{\log(2d/k)}{k}}.
\]
Equivalently, there is a subspace $F\subseteq\R^d$ with $\codim F<k$ such
that
\[
  \sup_{x\in K_1\cap F}\|x\|_2
  \lesssim
  \sqrt{\frac{\log(2d/k)}{k}}.
\]
\end{theorem}

\begin{proof}
We apply Gordon's escape theorem. If
$k \le m_{\mathrm E}$, then combining Proposition~\ref{prop:covariance-width-section} ($r_1(K_1) \lesssim \wg(K_1)$), Lemma~\ref{lem:K1-subset-B2} ($K_1 \subseteq B^d_2$) and Proposition~\ref{prop:localized-width} for $r=1$ gives
\[
  r_k(K_1)\le r_1(K_1)\lesssim \wg(K_1) = \wg(K_1 \cap B^d_2) \lesssim \sqrt{\log(1+d)} \le \sqrt{\log(2d)} \lesssim \sqrt{\frac{\log(2d/k)}{k}}.
\]
Assume $k \ge m_{\mathrm E} + 1$. Take $r:=\sqrt{\lambda \tfrac{\log(2d/k)}{k}}$
where $\lambda$ is a constant to be determined and let $S:=(\tfrac{1}{r} K_1) \cap \partial B_2^d$. Corollary~\ref{cor:localized-width-escape-scale}
yields a constant $C > 0$ so that
\[
  \wg(S)^2
  \le C \frac{k\log(2\lambda)}{\lambda}.
\]
Taking $\lambda$ large enough so that $\frac{C \log(2\lambda)}{\lambda} \le \frac{1}{2m_{\mathrm E}}$, we obtain
\[
  k-1\ge m_{\mathrm E}\max\{1,\wg(S)^2\}.
\]
Theorem~\ref{thm:escape} gives $F\subseteq\R^d$ with $\codim F\le k-1<k$
and $S\cap F=\emptyset$.  
If $K_1 \cap F$ contained a point $x$ with $\|x\|_2 \ge r$, then by convexity of $K_1$, the
point $\frac{r}{\|x\|_2} x$ would lie in $K_1\cap r\partial B^d_2\cap F = r(S \cap F) = \emptyset$, absurd.
Thus 
\[
r_k(K_1) \le \sup_{x\in K_1\cap F}\|x\|_2\le r \lesssim \sqrt{\frac{\log(2d/k)}{k}}. \qedhere
\]
\end{proof}

\section{A section radius bound for linear images of $K_1$}

Next, we derive an upper bound on the radii $r_k(MK_1)$. We again apply Gordon's escape theorem, but the localized width bound is more subtle; first we need to interpolate between $K_1$ and the Euclidean ball.

\begin{lemma}\label{lem:interpolation-body}
Let $p \in (1,2]$, and put $q=p/(p-1)$.  For every integer $k \in [1, d]$,
\[
  r_k \bigl(K_1,K_p\bigr)
  \lesssim
  \Big(\frac{\log(2d/k)}{k}\Big)^{1/q}.
\]
\end{lemma}

\begin{proof}
Theorem~\ref{thm:special-lewis} yields a subspace
$F\subseteq\R^d$ with $\codim F<k$ such that for $x \in F$,
\[
  \|x\|_2\lesssim \sqrt{\frac{\log(2d/k)}{k}} \|x\|_{K_1}.
\]
Lemma~\ref{lem:weighted-lewis-local-interpolation} gives
\[
  \|x\|_{K_p}
  \le \|x\|_{K_1}^{1-2/q}\|x\|_2^{2/q}
  \le \Big(\frac{\log(2d/k)}{k}\Big)^{1/q}\|x\|_{K_1}.
\]
Taking the supremum over $x\in K_1\cap F$ gives the claim.
\end{proof}

\begin{lemma}\label{lem:relative-entropy-MK1-MKp}
Let $p \in (1,2]$, and put $q=p/(p-1)$.  For every integer $k \in [1, d]$,
\[
  e_k \bigl(K_1,K_p\bigr)
  \lesssim
  \Big(\frac{\log(2d/k)}{k}\Big)^{1/q}.
\]
\end{lemma}

\begin{proof}
Since $1<p\le 2$, we have $q\ge 2$. By Lemma~\ref{lem:interpolation-body}, for every $1\le s\le k$,
\[
  r_s(K_1,K_p)
  \lesssim
  \left(\frac{\log(2d/s)}{s}\right)^{1/q}.
\]
Carl's inequality (Theorem~\ref{thm:carl-covering}) for $\alpha = 1$ gives
\[
  k e_k(K_1,K_p)
  \lesssim
  \max_{s \in [1,k]}
  s\left(\frac{\log(2d/s)}{s}\right)^{1/q}.
\]
Writing $u=k/s\ge1$, the last expression is
\[
  k^{1-1/q}u^{-(1-1/q)}(\log(2d/k)+\log u)^{1/q}.
\]
Since $\log(2d/k)\ge1$ and $q\ge2$,
\[
  u^{-(1-1/q)}(\log(2d/k)+\log u)^{1/q}
  \lesssim \log(2d/k)^{1/q},
\]
uniformly for $u\ge1$. Therefore
\[
  k e_k(K_1,K_p)
  \lesssim
  k^{1-1/q}\log(2d/k)^{1/q},
\]
and hence
\[
  e_k(K_1,K_p)
  \lesssim
  \left(\frac{\log(2d/k)}{k}\right)^{1/q}. \qedhere
\]
\end{proof}

\begin{lemma}\label{lem:Kp-gamma}
Let $p \in (1,2]$, and put $q=p/(p-1)$. Then
\[
  \Gamma\bigl(K_p\bigr)\lesssim \sqrt{q}.
\]
\end{lemma}

\begin{proof}
Let $Ux=(\ip{x}{u_1},\ldots,\ip{x}{u_m})$ and define the norm
\(\|v\|:=\left(\sum_{i=1}^m c_i|v_i|^p\right)^{1/p}\).
Then $\|x\|_{K_p}=\|Ux\|$.  Its dual norm is given by 
\(\|z\|_*:= \left(\sum_{i=1}^m c_i^{1-q}|z_i|^{q}\right)^{1/q}\).
We may write the polar body as
\[
  K_p^\circ
  =\left\{U^\top z:\|z\|_*\le1\right\}
  =\left\{\sum_{i=1}^m z_i u_i:
  \sum_{i=1}^m c_i^{1-q}|z_i|^{q}\le1\right\}.
\]

Take $P\in\pazocal M(K_p^\circ)$ attaining $\Gamma(K_p)$.  By Carath\'eodory's theorem, we may write
\[
  P=\sum_{\ell=1}^N\lambda_\ell y_\ell y_\ell^\top,
  \qquad
  y_\ell=U^\top z^{(\ell)},
  \qquad
  \|z^{(\ell)}\|_*\le1,
\]
where $\lambda_\ell\ge0$ and $\sum_{\ell=1}^N \lambda_\ell=1$.  Let
$g_1,\ldots,g_N$ be independent standard Gaussians.  Then $G \sim N(0, P)$ has the same
distribution as
\(\sum_{\ell=1}^N\sqrt{\lambda_\ell}g_\ell y_\ell =U^\top v\), where we denote \(v_i:=\sum_{\ell=1}^N\sqrt{\lambda_\ell}g_\ell z_i^{(\ell)}\). Also denote its covariance by $\sigma_i^2 := \sum_{\ell=1}^N\lambda_\ell(z_i^{(\ell)})^2$.
The vector
$U^\top v$ belongs to $\|v\|_*K_p^\circ$, so that
\[
  \|U^\top v\|_{K_p^\circ}\le \|v\|_*.
\]
Therefore, by concavity of $t \mapsto t^{1/q}$ and the Gaussian moment estimate,
\[
  \Gamma(K_p) = \E_{G \sim N(0, P)} \|G\|_{K_p^\circ}
  \le \E \|v\|_*
  \le
  \left(\E\sum_{i=1}^m c_i^{1-q}|v_i|^{q}\right)^{1/q} \lesssim
  \sqrt{q}
  \left(\sum_{i=1}^m c_i^{1-q}
  \sigma_i^q
  \right)^{1/q}.
\]
Since
$\sigma_i^2=\sum_{\ell=1}^N\lambda_\ell (z_i^{(\ell)})^2$,
the triangle inequality in the $\ell_{q/2}$ norm yields
\[
\begin{aligned}
  \left(\sum_{i=1}^m c_i^{1-q}\sigma_i^q\right)^{2/q}
  &=
  \left(\sum_{i=1}^m c_i^{1-q}
  \left|\sum_{\ell=1}^N \lambda_\ell (z_i^{(\ell)})^2\right|^{q/2}
  \right)^{2/q} \\
  &\le
  \sum_{\ell=1}^N \lambda_\ell
  \left(\sum_{i=1}^m c_i^{1-q}|z_i^{(\ell)}|^q\right)^{2/q} \\
  &=
  \sum_{\ell=1}^N \lambda_\ell \|z^{(\ell)}\|_*^2 \le 1.
\end{aligned}
\]
We conclude $\Gamma(K_p) \lesssim \sqrt{q}$, as claimed.
\end{proof}

\begin{proposition}\label{prop:localized-image-width}
Let $M:\R^d\to\R^n$ be linear.  For every integer $k\in[1,d]$ and every $r \ge 0$,
\[
  \wg(MK_1\cap rB_2^n)
  \lesssim
  r\sqrt k+ \norm{M}_{\pazocal{B}_n(K_1)}\sqrt{\log(2d/k)}.
\]
\end{proposition}

\begin{proof}
Put $q:=2\log(2d/k) \ge 2,  p:=\frac{q}{q-1} \in (1,2]$. By Lemma~\ref{lem:relative-entropy-MK1-MKp}, for
$t \asymp\Big(\frac{\log(2d/k)}{k}\Big)^{1/q}$ one has
\[
 N(MK_1 \cap rB^n_2, t MK_p) \le N(MK_1, t MK_p)\le N(K_1, t K_p) \le 2^{k-1},
\]
where we have used that $N(MK, ML) \le N(K,L)$ for any convex bodies $K, L$ as we may simply apply $M$ to any cover.
Lemma~\ref{lem:cover-to-width}, applied with $L:=MK_p$, gives
\[
  \wg(MK_1 \cap rB^n_2)\lesssim r\sqrt k+t\,\wg(MK_p).
\]
By Lemma~\ref{lem:weighted-lewis-l1-lp-comparison}, $K_p\subseteq d^{1/q}K_1$.  Thus $K_1^\circ\subseteq d^{1/q}K_p^\circ$, so $\pazocal{B}_n(K_1) \subseteq d^{1/q} \pazocal{B}_n(K_p)$, and therefore
$\norm{M}_{\pazocal{B}_n(K_p)}\le d^{1/q} \norm{M}_{\pazocal{B}_n(K_1)} $.  Lemmas~\ref{lem:width-covariance-domination} and \ref{lem:Kp-gamma} thus yield
\[
  \wg(MK_p)
  \le \norm{M}_{\pazocal{B}_n(K_p)}\Gamma(K_p)
  \lesssim \norm{M}_{\pazocal{B}_n(K_1)}\,d^{1/q}\sqrt q.
\]
Therefore
\[
  t\,\wg(MK_p)
  \lesssim
  \norm{M}_{\pazocal{B}_n(K_1)}\sqrt q\left(\frac{d\log(2d/k)}{k}\right)^{1/q} \lesssim \norm{M}_{\pazocal{B}_n(K_1)}\sqrt{\log(2d/k)}.
\]
This proves the claim.
\end{proof}

\begin{theorem}\label{thm:main-lewis-geometric}
Let $M:\R^d\to\R^n$ be linear.  For every integer $k \in [1,d]$,
\[
  r_k\bigl(MK_1\bigr)
  \lesssim
  \norm{M}_{\pazocal{B}_n(K_1)}
  \sqrt{\frac{\log(2d/k)}{k}}.
\]
\end{theorem}
\begin{proof}

If \(\norm{M}_{\pazocal B_n(K_1)}=0\), then \(MK_1=\{0\}\) by
Lemma~\ref{lem:Bnorm-dominates-radius-one}, and the claim is trivial, so we may assume the norm is positive. Again by Lemma~\ref{lem:Bnorm-dominates-radius-one} we have $r_k(MK_1)\le r_1(MK_1)\le \norm{M}_{\pazocal{B}_n(K_1)}$, so the bound is immediate for $k = O(1)$, and we have the freedom to assume $k$ is larger than any chosen constant.

Fix an absolute constant $C_0$ for which Proposition~\ref{prop:localized-image-width} gives
\[
  \wg(MK_1\cap rB_2^n)
  \le C_0\left(r\sqrt j+\norm{M}_{\pazocal{B}_n(K_1)}\sqrt{\log(2d/j)}\right)
\]
for all integer $j \in [1, d]$ and $r \ge 0$.  Let $m_{\mathrm E}$ be the constant in Theorem~\ref{thm:escape}.  Choose universal constants $0<\alpha <1/2$ and $\beta \ge1$ so that
\[
  C_0\sqrt{2\alpha}
  +
  \frac{C_0}{\beta}
  \sqrt{1+\log(2/\alpha)}
  \le \frac{1}{\sqrt{2m_{\mathrm E}}}.
\]
Put $j:=\lfloor\alpha k\rfloor$ and
\[
  r:=\beta \norm{M}_{\pazocal{B}_n(K_1)}\sqrt{\frac{\log(2d/k)}{k}}.
\]
Increasing the universal lower bound on $k$ if necessary, we may assume
$1\le j\le d$, $j\le 2\alpha k$, and $j\ge \alpha k/2$. Let $H:=\spanop(MK_1)$ and
\[
  S:=\left(\frac1r MK_1\right)\cap \partial B_2^H.
\]
Then $S\subseteq r^{-1}(MK_1\cap rB_2^n)$, and Proposition~\ref{prop:localized-image-width} gives
\[
\begin{aligned}
  \wg(S)
  &\le r^{-1}\wg(MK_1\cap rB_2^n) \\
  &\le C_0\sqrt j+
  C_0\frac{\norm{M}_{\pazocal{B}_n(K_1)}\sqrt{\log(2d/j)}}{r}.
\end{aligned}
\]
Since $j\le2\alpha k$ and, using $j\ge\alpha k/2$,
\[
  \frac{\log(2d/j)}{\log(2d/k)}
  =1+\frac{\log(k/j)}{\log(2d/k)}
  \le 1+\log(2/\alpha),
\]
we get
\[
  \wg(S)
  \le \frac{\sqrt{k}}{\sqrt{2m_{\mathrm E}}}.
\]
For $k$ large enough, this implies
\[
  k-1\ge m_{\mathrm E}\max\{1,\wg(S)^2\}.
\]
 Theorem~\ref{thm:escape} yields a subspace
$F\subseteq H$ with $\codim_HF \le k-1 < k$ and $S \cap F =\varnothing$.

If $MK_1 \cap F$ contained a point $x$ with $\|x\|_2 \ge r$, then by convexity of $MK_1$, the
point $\frac{r}{\|x\|_2} x$ would lie in $MK_1 \cap r\partial B^H_2\cap F$, so $\frac{x}{\|x\|_2} \in S \cap F$, a contradiction.
Hence 
\[
r_k(MK_1) \le \sup_{x\in MK_1 \cap F}\|x\|_2\le r \lesssim   \norm{M}_{\pazocal{B}_n(K_1)}
  \sqrt{\frac{\log(2d/k)}{k}}.
  \qedhere
\]
\end{proof}

\section{From section radii to volume bounds}

Theorem~\ref{thm:main-lewis-geometric} is invariant under linear transformations, which implies the following corollary for arbitrary quotients of sections of the $\ell_1$ ball:

\begin{corollary}\label{cor:b1-section-estimate}
Let $F\subseteq\R^m$ be $d$-dimensional and let $M:\R^m \to\R^n$ be linear.  Then,
for every integer $k \in [1,d]$,
\[
  r_k(M(B_1^m\cap F))
  \lesssim \norm{M}_{\pazocal{B}_n(B_1^m\cap F)}\sqrt{\frac{\log(2d/k)}{k}}.
\]
\end{corollary}

\begin{proof}
Take a linear isomorphism $A:\R^d\to F$ and set
\(K:=\{x\in\R^d:\norm{Ax}_1\le1\}\).
Thus $AK=B_1^m\cap F$. Theorem~\ref{thm:lewis-rep} gives an invertible $T$ with
$TK$ in Lewis position.  Applying Theorem~\ref{thm:main-lewis-geometric} to $TK$,
\[
  r_k(M(B_1^m\cap F))=r_k(MAT^{-1} TK)
  \lesssim
  \norm{MAT^{-1}}_{\pazocal B_n(TK)}
  \sqrt{\frac{\log(2d/k)}{k}}.
\]
It is straightforward to check that by definition
$\norm{MAT^{-1}}_{\pazocal B_n(TK)}=\norm{M}_{\pazocal B_n(AK)}$.
\end{proof}

Now we are ready to start moving towards volumetric bounds. We first pass to entropy numbers via Carl's inequality, then bound volume through covering numbers.

\begin{corollary}\label{cor:b1-covering-estimate}
Let $F\subseteq\R^m$ be $d$-dimensional and let $M: \R^m \to\R^n$ be linear.  Then,
for every integer $k \in [1,d]$,
\[
  e_k(M(B_1^m\cap F))
  \lesssim \norm{M}_{\pazocal{B}_n(B_1^m\cap F)}\sqrt{\frac{\log(2d/k)}{k}}.
\]
\end{corollary}

\begin{proof}

By Corollary~\ref{cor:b1-section-estimate}, for every $j \in [1,k]$, we have
\[
  j r_j(M(B_1^m\cap F))
  \lesssim \norm{M}_{\pazocal{B}_n(B_1^m\cap F)}\sqrt{j\log(2d/j)}.
\]
Since $j \mapsto j \log(2d/j)$ attains its maximum either at $j = k$ or at one of the integers closest to $2d/e$ if $k > 2d/e$, it follows that $j \log(2d/j) \lesssim k \log(2d/k)$. Therefore, by Theorem~\ref{thm:carl-covering} with $K=M(B_1^m\cap F)$,
$L=B_2^n$, and $\alpha=1$,
\[
  k e_k(M(B_1^m\cap F))
  \lesssim \max_{j \in [1, k]} j r_j(M(B_1^m\cap F)) \lesssim \norm{M}_{\pazocal{B}_n(B_1^m\cap F)}\sqrt{k\log(2d/k)}. \qedhere
\]
\end{proof}

\begin{lemma}\label{lem:john-Bnorm-geometric}
Let $K \subseteq \R^m$ be a symmetric convex set with span $H$, and let $M: \R^m \to\R^n$ be linear.  Assume that $MK$ is
full-dimensional in $\R^n$ and that its maximal-volume inscribed ellipsoid is
$B_2^n$.  Then
\[
  \norm{M}_{\pazocal{B}_n(K)}\le \sqrt n.
\]
\end{lemma}

\begin{proof}
By Theorem~\ref{thm:john-loewner}, there exist points $u_1,\ldots,u_N\in (MK)^\circ\cap \partial B_2^n$ and weights
$c_1,\ldots,c_N>0$ such that
\[
  \sum_{i=1}^N c_i u_i u_i^\top=I_n.
\]
Taking traces gives $\sum_{i=1}^N c_i=n$. For each $i \in [N]$, $M^\top u_i\in K^\circ$, because
\[
  \sup_{x\in K}\abs{\ip{x}{M^\top u_i}}
  =\sup_{z\in MK}\abs{\ip{z}{u_i}}\le1.
\]
Therefore
\(P:=\sum_{i=1}^N \frac{c_i}{n} (M^\top u_i)(M^\top u_i)^\top \in \pazocal M(K^\circ)\),
and
\[
  M^\top M
  =M^\top I_n M
  =\sum_{i=1}^N c_i (M^\top u_i)(M^\top u_i)^\top
  =nP.
\]
Thus $M^\top M = (\sqrt{n})^2P$ with $P \in \pazocal M(K^\circ)$, so that $\norm{M}_{\pazocal{B}_n(K)}\le\sqrt{n}$.
\end{proof}

We now prove a volume-ratio bound for quotients of sections of $B_1^m$.

\begin{theorem}\label{thm:vr-quotient-geometric}
Let $F\subseteq\R^m$ be a $d$-dimensional subspace. Let $M:\R^m\to\R^n$ be a linear map with $1\le n\le d \le m$. Assume that $M(B_1^m\cap F)$ is full-dimensional
in $\R^n$. Then
\[
  \vr\bigl(M(B_1^m\cap F)\bigr)\lesssim \sqrt{\log\Big(\frac{2d}{n}\Big)}.
\]
\end{theorem}

\begin{proof}[Proof of Theorem~\ref{thm:vr-quotient-geometric}]
We may assume without loss of generality that the John ellipsoid of $M(B_1^m\cap F)$
is $B_2^n$, otherwise apply an invertible linear transformation.  By Lemma~\ref{lem:john-Bnorm-geometric},
$\norm{M}_{\pazocal{B}_n(B_1^m\cap F)}\le\sqrt n$. Corollary~\ref{cor:b1-covering-estimate} with $k=n$ gives
\[
  e_n(M(B_1^m\cap F))
  \lesssim \sqrt n\sqrt{\frac{\log(2d/n)}{n}}
  = \sqrt{\log(2d/n)}.
\]
Thus, for every $t>e_n(M(B_1^m\cap F))$, $N\left(M(B_1^m\cap F),t B_2^n\right)\le2^{n-1}$. Therefore,
\[
  \vol_n\bigl(M(B_1^m\cap F)\bigr)
  \le2^{n-1}t^n\vol_n(B_2^n).
\]
Taking $n$th roots yields
\[
   \vr\bigl(M(B_1^m\cap F)\bigr) = \left(\frac{\vol_n\bigl(M(B_1^m\cap F)\bigr)}{\vol_n(B_2^n)}\right)^{1/n}
  \le 2e_n(M(B_1^m\cap F))
  \lesssim \sqrt{\log(2d/n)}. \qedhere
\]
\end{proof}

The form needed for vector balancing is the following volume bound.

\begin{corollary}\label{cor:qvr}
Let $F\subseteq\R^m$ be a $d$-dimensional subspace.  Let
$M:\R^m\to\R^n$ be a linear map whose corresponding matrix has coefficients in $[-1,1]$, with $1 \le n \le d$. Then
\[
  \vol_n\bigl(M(B_1^m\cap F)\bigr)^{1/n}
  \lesssim \sqrt{\frac{\log(2d/n)}{n}}.
\]
\end{corollary}

\begin{proof}[Proof of Corollary~\ref{cor:qvr}]
If the image $M(B_1^m\cap F)$ is not full-dimensional, then its $n$-dimensional volume is zero and the conclusion is immediate.  Assume from now on that the image is full-dimensional.

By Theorem~\ref{thm:vr-quotient-geometric}, if $\pazocal E$ is the John
ellipsoid of $M(B_1^m\cap F)$, then
\begin{equation}\tag{$\ast$}\label{eq:vr-main-use}
  \left(\frac{\vol_n\bigl(M(B_1^m\cap F)\bigr)}{\vol_n(\pazocal E)}\right)^{1/n}
  \lesssim \sqrt{\log(2d/n)}.
\end{equation}
The assumption on $M$ implies $M(B_1^m\cap F)\subseteq B_\infty^n$: if
$b\in B_1^m\cap F$, then for each $j$,
\[
  \abs{(Mb)_j}=\abs{\ip{m_j}{b}}
  \le \norm{m_j}_\infty\norm{b}_1\le1.
\]
Thus $\pazocal E\subseteq B_\infty^n$. Write $\pazocal E = QB_2^n$, with $Q$
invertible, and write the rows of $Q$ as $q_1,\ldots,q_n$.  The inclusion
$QB_2^n\subseteq B_\infty^n$ implies that every $q_i$ has Euclidean norm at
most $1$, because
\[
  \norm{q_i}_2= \sup_{\norm{x}_2\le1}\abs{\ip{q_i}{x}}\le1.
\]
Hadamard's determinant inequality gives
\[
  \abs{\det Q}\le\prod_{i=1}^n\norm{q_i}_2\le1.
\]
Therefore $\vol_n(QB_2^n)\le\vol_n(B_2^n)$. Using the standard estimate
$\vol_n(B_2^n)^{1/n}\lesssim 1/\sqrt n$ and combining with \eqref{eq:vr-main-use} yields
\[
  \vol_n\bigl(M(B_1^m\cap F)\bigr)^{1/n}
  \lesssim \sqrt{\log(2d/n)}\,\vol_n(\pazocal E)^{1/n}
  \lesssim \sqrt{\frac{\log(2d/n)}{n}}. \qedhere
\]
\end{proof}

\section{Balancing vectors in any zonotope}

We now proceed towards the proof of Theorem~\ref{thm:main-vb}.  Let
$Z = A^\top B_\infty^m \subseteq\R^d$ be a zonotope, which we assume without loss of generality is full-dimensional (otherwise we may reduce $d$ after an invertible linear transformation) so that $A:\R^d\to\R^m$ has rank $d$, and set
$F:=A\R^d\subseteq\R^m$.

Let $v_1,\ldots,v_n\in Z$, where $n \in [1,d]$, and take
$u_1,\ldots,u_n\in B_\infty^m$ with $v_i=A^\top u_i$ for $i \in [n]$.  For
$S\subseteq[n]$, define $M_S:\R^m\to\R^S$ by
$(M_Sb)_i:=\ip{u_i}{b}$ and $V_S:\R^d\to\R^S$ by
$(V_Sx)_i:=\ip{v_i}{x}$, for $i\in S$.  Then $V_S=M_SA$ and
$V_S^\top a=\sum_{i\in S}a_i v_i=A^\top M_S^\top a$.  Define the coordinate
body
\[
  K_S:=\{a\in\R^S:V_S^\top a\in Z\}.
\]
Equivalently, if $V:\R^d\to\R^n$ is given by $(Vx)_i=\ip{v_i}{x}$ and
$K:=\{a\in\R^n:V^\top a\in Z\}$, then $K_S = K \cap \R^S$.

\begin{lemma}\label{lem:polar}
For every nonempty $S\subseteq[n]$,
\[
  K_S=\bigl(M_S(B_1^m\cap F)\bigr)^\circ.
\]
\end{lemma}

\begin{proof}
The polar of $Z=A^\top B_\infty^m$ is $Z^\circ=\{x\in\R^d:\norm{Ax}_1\le1\}$. Indeed,
\[
  \sup_{z\in B_\infty^m}\ip{A^\top z}{x}
  =\sup_{z\in B_\infty^m}\ip{z}{Ax}
  =\norm{Ax}_1.
\]
Since $A$ has rank $d$, the map $x\mapsto Ax$ identifies
$Z^\circ$ with $B_1^m\cap F$.

For $y\in\R^S$,
\begin{align*}
  \norm{V_S^\top y}_Z
  &=\sup_{x\in Z^\circ}\abs{\ip{V_S^\top y}{x}} \\
  &=\sup_{x:\norm{Ax}_1\le1}\abs{\ip{y}{V_Sx}} \\
  &=\sup_{x:\norm{Ax}_1\le1}\abs{\ip{y}{M_SAx}} \\
  &=\sup_{b\in B_1^m\cap F}\abs{\ip{y}{M_Sb}}.
\end{align*}
Thus $y\in K_S$ iff $\abs{\ip{y}{z}}\le1$ for every $z\in M_S(B_1^m\cap F)$, i.e. $y\in\bigl(M_S(B_1^m\cap F)\bigr)^\circ$.
\end{proof}

\begin{proposition}\label{prop:vl}
For every nonempty
$S\subseteq[n]$, $k:=|S|$, one has
\[
  \vol_k(K_S)^{1/k}
  \gtrsim
  \frac{1}{\sqrt{k\log(2d/k)}}.
\]
\end{proposition}

\begin{proof}
Fix $S\subseteq[n]$ and put $k=|S|$. If $M_S(B_1^m\cap F)$ is lower-dimensional, then Lemma~\ref{lem:polar} implies that
$K_S$ is unbounded, and the claim is trivial. Now assume that $M_S(B_1^m\cap F)$ is full-dimensional.  By Lemma~\ref{lem:polar} we have $K_S=(M_S(B_1^m\cap F))^\circ$.
By Corollary~\ref{cor:qvr}, since $u_1, \dots, u_n \in B_\infty^m$,
\[
  \vol_k(M_S(B_1^m\cap F))^{1/k}
  \lesssim \sqrt{\frac{\log(2d/k)}{k}}.
\]
Therefore, the inverse Santal\'o inequality (Theorem~\ref{thm:inverse-santalo}) yields
\[
  \vol_k(K_S)^{1/k}
  =
  \vol_k((M_S(B_1^m\cap F))^\circ)^{1/k}
  \gtrsim
  \frac{1}{\sqrt{k\log(2d/k)}}. \qedhere
\]
\end{proof}

We now convert Proposition~\ref{prop:vl} into a partial coloring.

\begin{lemma}\label{lem:active-pc}
For every $y\in[-1,1]^n$, there exists a polynomial-time computable $y'\in[-1,1]^n$ such that at least $n/2$ coordinates of $y'$ lie in
$\{\pm1\}$ and
\[
  \left\|\sum_{i=1}^n (y_i-y'_i)v_i\right\|_Z
  \lesssim
  \sqrt{n\log(2d/n)}.
\]
\end{lemma}

\begin{proof}
Apply Proposition~\ref{prop:vl} and Corollary~\ref{thm:volume-pc} after scaling $K$ by $\sqrt{n \log(2d/n)}$.
\end{proof}

\begin{proof}[Proof of Theorem~\ref{thm:main-vb}]
Set $y^{(0)}=\mathbf{0}$.  Given $y^{(j)}$, let $I_j$ be the set of coordinates not yet in
$\{\pm1\}$ and put $n_j=|I_j|$.  If $n_j>0$, apply
Lemma~\ref{lem:active-pc} to the active vectors $(v_i)_{i\in I_j}$ and the
point $y^{(j)}|_{I_j}$.  Keeping the fixed coordinates unchanged gives
$y^{(j+1)}\in[-1,1]^n$ with $n_{j+1}\le n_j/2$ and
\[
  \left\|\sum_{i=1}^n\bigl(y_i^{(j)}-y_i^{(j+1)}\bigr)v_i\right\|_Z
  \lesssim
  \sqrt{n_j\log(2d/n_j)}.
\]
Let $t$ be the first index with $n_t=0$.  Such a $t$ exists since the active
set size is at least halved at each step.  Then all coordinates of
$y^{(t)}$ are signs; write $x_i=y_i^{(t)}$.

By the triangle inequality and symmetry of $Z$,
\[
  \left\|\sum_{i=1}^n x_i v_i\right\|_Z
  \le
  \sum_{j=0}^{t-1}
  \left\|\sum_{i=1}^n\bigl(y_i^{(j+1)}-y_i^{(j)}\bigr)v_i\right\|_Z
  \lesssim
  \sum_{j=0}^{t-1}\sqrt{n_j\log(2d/n_j)}.
\]
Since \(n_j\le n2^{-j}\) for $j \in [0,t)$,
\[
  n_j\log(2d/n_j)
  \lesssim
  n2^{-j}\log(2d2^j/n)
  =
  n2^{-j}(\log(2d/n)+j)
  \le
  n\log(2d/n)\,2^{-j}(j+1).
\]
Thus the sum may be upper bounded by
\[
  \sum_{j=0}^{t-1}\sqrt{n_j\log(2d/n_j)}
  \lesssim
  \sqrt{n\log(2d/n)}\sum_{j\ge0}2^{-j/2}\sqrt{j+1}
  \lesssim
  \sqrt{n\log(2d/n)}.
\]
Hence the signed sum lies in
$C\sqrt{n\log(2d/n)}\,Z$ for a universal constant $C$.
\end{proof}

\section{Open problems}\label{sec:open-problems}

We end with two natural open problems. The first is a variant of Theorem~\ref{thm:vr-quotient-geometric} that would yield a $O(\sqrt{d})$ partial coloring bound for the Matrix Spencer setting.

\begin{conjecture} Is it true that for any linear map $M:\R^{d \times d} \to\R^d$ one has 
\[\vr(M\pazocal S^d_1) \lesssim 1,
\]
where $\pazocal S^d_1 := \{A \in \mathbb{R}^{d\times d} : \|A\|_* \le 1\}$ is the Schatten-1 ball?
\end{conjecture}

Finally, we restate a version of Schechtman's 1987 sparsification problem. 

\begin{conjecture}\label{conj:schechtman-sparsification}
Is there a universal constant $C$ such that the following holds?  Let
$Z\subseteq\R^d$ be a $d$-dimensional zonotope and let $\eps \in (0,1)$.  Then there is
a zonotope $\widetilde Z = \tilde{A}^\top [-1,1]^{\widetilde{m}} \subseteq\R^d$ generated by $\widetilde{m} \le C d/\eps^2$ segments such that
\[
  \widetilde Z\subseteq Z\subseteq (1+\eps)\widetilde Z.
\]
\end{conjecture}
\section{Acknowledgments}

The author used GPT-5.5 Pro during the development of this work to explore proof
strategies, translate between operator theory and convex geometry, and assist with verification. Every AI-generated proof was verified and rewritten by the author, who takes full responsibility for the paper.

\appendix

\section{Proofs for the covariance-width section estimates}
\label{app:covariance-width-section-proofs}

\begin{proof}[Proof of Proposition~\ref{prop:covariance-width-section}]
Let $H:=\spanop K$ and take $x_0\in K$ with $\|x_0\|_2= \sup_{x \in K} \|x\|_2$.  Since $K$ is symmetric,
\[
  \wg(K)=\E\sup_{x\in K}\ip{g}{x}
  \ge \E_{g \sim N(0, I_H)} |\ip{g}{x_0}|
  =  \|x_0\|_2 \cdot \E_{\gamma \sim N(0,1)} |\gamma|.
\]
Thus $r_1(K) = \|x_0\|_2 \lesssim \wg(K)$. If
$k \le m_{\mathrm E}$, then
\[
  r_k(K)\le r_1(K)\lesssim \wg(K)\lesssim \frac{\wg(K)}{\sqrt{k}}.
\]
Assume $k > m_{\mathrm E}$. We will apply Gordon's escape theorem. Set $r:=\sqrt{2m_{\mathrm E}} \frac{\wg(K)}{\sqrt{k}}$ and \[S:=(\tfrac{1}{r} K) \cap \partial B^H_2 \subseteq \partial B^H_2.\] Then $\wg(S)\le r^{-1}\wg(K)$, and by the choice of $r$ we have
$k-1\ge m_{\mathrm E}\max\{1,\wg(S)^2\}$. Theorem~\ref{thm:escape} yields a subspace
$F\subseteq H$ with $\codim_HF \le k-1 < k$ and $S \cap F =\varnothing$.

If $K\cap F$ contained a point $x$ with $\|x\|_2 \ge r$, then by convexity of $K$, the
point $\frac{r}{\|x\|_2} x$ would lie in $K \cap r\partial B^H_2\cap F$, so $\frac{x}{\|x\|_2} \in S \cap F$, a contradiction.
Thus $r_k(K) \le \sup_{x\in K\cap F}\|x\|_2\le r$.
\end{proof}

\begin{proof}[Proof of Lemma~\ref{lem:Bnorm-convex}]
Let $M,N\in \pazocal{B}_n(K)$ and $\theta \in [0,1]$. Take
$A,B\in\pazocal M(K^\circ)$ such that $M^\top M\preceq A$, $N^\top N\preceq B$. For every $x\in\R^d$,
\[
  \|\bigl(\theta M+(1-\theta)N\bigr)x\|_2
  \le \theta\|Mx\|_2+(1-\theta)\|Nx\|_2
  \le \theta \sqrt{x^\top A x}+(1-\theta)\sqrt{x^\top Bx}.
\]
By Cauchy-Schwarz,
\[
  \left(\theta \sqrt{x^\top A x}+(1-\theta)\sqrt{x^\top Bx}\right)^2
  \le (\theta + 1-\theta) \cdot x^\top\bigl(\theta A+(1-\theta)B\bigr)x.
\]
Since $\pazocal M(K^\circ)$ is convex,
$C:=\theta A+(1-\theta)B$ belongs to $\pazocal M(K^\circ)$.  Hence
\[
  \bigl(\theta M+(1-\theta)N\bigr)^\top
  \bigl(\theta M+(1-\theta)N\bigr)
  \preceq C \implies \theta M+(1-\theta)N\in\pazocal{B}_n(K). \qedhere
\]
\end{proof}

\begin{proof}[Proof of Lemma~\ref{lem:Bnorm-dominates-radius-one}]
Let $t = \norm{M}_{\pazocal{B}_n(K)}$ and take $A\in\pazocal M(K^\circ)$ such that
$M^\top M\preceq t^2A$.  Write
$A=\sum_i\lambda_i y_i y_i^\top$ with $y_i\in K^\circ$, $\lambda \ge \mathbf{0}$, and
$\sum_i\lambda_i=1$.  If $x\in K$, then
\[
  \|Mx\|_2^2
  =\ip{x}{M^\top Mx}
  \le t^2\ip{x}{Ax}
  =t^2\sum_i\lambda_i|\ip{x}{y_i}|^2
  \le t^2,
\]
since $y_i\in K^\circ$.  Hence $r_1(MK) = \sup_{x\in K}\|Mx\|_2\le t= \norm{M}_{\pazocal{B}_n(K)}$.
\end{proof}

\begin{proof}[Proof of Lemma~\ref{lem:width-covariance-domination}]
Let $t=\norm{M}_{\pazocal{B}_n(K)}$ and take $A\in\pazocal M(K^\circ)$ such that
  $M^\top M\preceq t^2A$. Let $g \sim N(0, I_n)$.  Then \[\wg(MK)=\E\sup_{x\in K}\ip{g}{Mx}
  =\E\|M^\top g\|_{K^\circ}.\]
The vector $M^\top g$ is Gaussian with covariance $M^\top M$.  Let $\eta$ be an
independent Gaussian vector with covariance $t^2A-M^\top M$.  Then
$\zeta:=M^\top g+\eta$ is Gaussian with covariance $t^2A$.  Since
$\E(\zeta\mid M^\top g)=M^\top g$, Jensen's inequality gives
\[
  \E\|M^\top g\|_{K^\circ}
  =\E\|\E(\zeta\mid M^\top g)\|_{K^\circ}
  \le \E\|\zeta\|_{K^\circ}.
\]
Because $\zeta$ has the same distribution as $t \alpha$ with
$\alpha \sim N(0,A)$, we obtain
\[
  \wg(MK) =  \E\|M^\top g\|_{K^\circ} \le t\,\E_{\alpha \sim N(0,A)}\|\alpha\|_{K^\circ}
  \le t\,\Gamma(K) =  \norm{M}_{\pazocal{B}_n(K)} \Gamma(K). \qedhere
\]
\end{proof}

\begin{proof}[Proof of Lemma~\ref{lem:cover-to-width}]
Choose a cover $K\subseteq\bigcup_i(z_i+tL)$ with at most $2^{k-1}$ translates.
Discarding empty translates, choose $x_i\in K\cap(z_i+tL)$.  Since $L$ is symmetric,
\[
  K\subseteq \bigcup_i(x_i+2tL).
\]
Indeed, if $y,x_i\in z_i+tL$, then $y-x_i\in tL-tL=2tL$.  For
$g\sim N(0,I_n)$,
\[
  \sup_{y\in K}\ip{g}{y}
  \le \max_i\ip{g}{x_i}+2t\sup_{z\in L}\ip{g}{z}.
\]
It remains to estimate the discrete Gaussian maximum.  Let $N \le 2^{k-1}$ be the number
of chosen points.  For each $i$, the random variable $\ip{g}{x_i}$ is
centered Gaussian with variance $\|x_i\|_2^2\le r^2$.  We use the
 Gaussian moment-generating formula
\[
  \E e^{\lambda \ip{g}{x_i}}=\exp\left(\frac{\lambda^2\|x_i\|_2^2}{2}\right) \le \exp\left(\frac{\lambda^2 r^2}{2}\right)
\]
for $\lambda = \frac{\sqrt{2 \ln N}}{r}$; we obtain by convexity of $z \mapsto e^{\lambda z}$ that
\[
  \E\max_i \ip{g}{x_i}
  \le {\frac{1}{\lambda}}\ln\sum_{i=1}^N \E e^{\lambda \ip{g}{x_i}}
  \le {\frac{\ln N}{\lambda}}+{\frac{\lambda r^2}{2}} = r \cdot \sqrt{2\ln N} \lesssim r \sqrt{k}.
\]
Therefore, $\wg(K) \lesssim r \sqrt{k} + 2t \wg(L)$, as claimed.
\end{proof}

\end{document}